\documentclass{article}
\usepackage[utf8]{inputenc}
\usepackage[a4paper,margin=1.5cm,footskip=.5cm]{geometry}

\usepackage{amsmath, amsfonts, amssymb, amsthm} 
\usepackage[
backend=biber,
style=alphabetic,
]{biblatex}
\usepackage{mathrsfs}
\usepackage{mathtools}
\usepackage{enumitem} 
    \setlist[enumerate]{label=(\roman*),noitemsep,topsep=3pt} 
    \setlist[itemize]{label={$\cdot$},noitemsep,topsep=3pt} 
\usepackage{verbatim} 
\usepackage{hyperref} 
\usepackage{url}
\usepackage{theoremref} 
\usepackage{bbm}
\usepackage{color}
\usepackage{float}
\definecolor{gray}{rgb}{0.5,0.5,0.5}
\usepackage{subcaption}
\usepackage{adjustbox}
\usepackage{tikz}
    \usetikzlibrary{matrix}
    \usetikzlibrary{backgrounds}
    \usetikzlibrary{positioning}
    
    \newdimen\nodeSize
    \newdimen\nodeDist
    \tikzset{
        position/.style args={#1:#2 from #3}{
            at=(#3.#1), anchor=#1+180, shift=(#1:#2)
        }
    }

\usepackage{listings}

\usepackage{rotating}
\usepackage{makecell}
\usepackage{tabu}
\usepackage{multirow} 

\makeatletter
\newcommand\ackname{Acknowledgements}
\if@titlepage
  \newenvironment{acknowledgements}{%
      \titlepage
      \null\vfil
      \@beginparpenalty\@lowpenalty
      \begin{center}%
        \bfseries \ackname
        \@endparpenalty\@M
      \end{center}}%
     {\par\vfil\null\endtitlepage}
\else
  \newenvironment{acknowledgements}{%
      \if@twocolumn
        \section*{\abstractname}%
      \else
        \small
        \begin{center}%
          {\bfseries \ackname\vspace{-.5em}\vspace{\z@}}%
        \end{center}%
        \quotation
      \fi}
      {\if@twocolumn\else\endquotation\fi}
\fi
\makeatother

\newtheorem{theorem}{Theorem}[section]
\newtheorem{corollary}[theorem]{Corollary}
\newtheorem{proposition}[theorem]{Proposition}
\newtheorem{lemma}[theorem]{Lemma}

\theoremstyle{definition}
\newtheorem{definition}[theorem]{Definition}
\newtheorem{remark}[theorem]{Remark}
\newtheorem{problem}{Problem}
\newtheorem{example}[theorem]{Example}

\newcommand{\thistime}{\expandafter\calctimeA\pdfcreationdate\@nil} 
\def\calctimeA#1:#2#3#4#5#6#7#8#9{\calctimeB}
\def\calctimeB#1#2#3#4#5\@nil{#1#2:#3#4}

\newcommand{\zen}{\mathsf{Z}} 

\renewcommand{\S}{\mathsf{S}}
\newcommand{\C}{\mathsf{C}}
\newcommand{\D}{\mathsf{D}}

\newcommand{\GL}{\mathsf{GL}}
\newcommand{\SL}{\mathsf{SL}}
\newcommand{\K}{\mathsf{K}}
\newcommand{\W}{\mathsf{W}}

\newcommand{\numberset}[1]{\mathbbm{#1}}
\newcommand{\N}{\numberset{N}}
\newcommand{\Z}{\numberset{Z}}

\newcommand{\aut}{\operatorname{\mathsf{Aut}}}

\renewcommand{\gcd}{\operatorname{\mathsf{gcd}}}

\newcommand{\conj}{\operatorname{\mathsf{Conj}}}
\newcommand{\arcs}{\operatorname{\mathsf{Arcs}}}

\newcommand{\+}{\mkern2mu} 
\let\perogni\forall 
\let\esiste\exists
\renewcommand{\forall}{\perogni\+}
\renewcommand{\exists}{\esiste\+}

\newcommand{\s}{\sigma}

\newcommand{\mo}{^{-1}} 

\newcommand\Tstrut{\rule{0pt}{5ex}}         
\newcommand\Bstrut{\rule[-0.9ex]{0pt}{0pt}}   

\title{Coloring Torus Knots by Conjugation Quandles}
\author{Filippo Spaggiari\\ \small{\href{spaggiari@karlin.mff.cuni.cz}{spaggiari@karlin.mff.cuni.cz}}\\\small{Department of Algebra, Faculty of Mathematics \& Physics, Charles University Prague, Czech Republic}}
\date{\today}

\begin{document}

\maketitle

\begin{abstract}
    In the first part of this paper, we present general results concerning the colorability of torus knots using conjugation quandles over any abstract group. Subsequently, we offer a numerical characterization for the colorability of torus knots using conjugation quandles over some particular groups, such as the matrix groups $\GL(2,q)$ and $\SL(2,q)$, the dihedral group, and the symmetric group.
\end{abstract}

\tableofcontents

\section{Introduction} \label{sec:introduction}

Coloring invariants serve as valuable computational tools in various contexts (see \cite{Clark}, \cite{Fish}, and \cite{elhamdadi2015quandles}). In addition, Kuperberg's NP certificate of knottedness can be effectively interpreted using coloring techniques (see \cite{Kuperberg}). A significant class of examples related to coloring is provided by conjugation quandles. Notably, a nontrivial coloring achieved with a quandle $Q$ implies the existence of a nontrivial coloring with $Q/\lambda$, where $\lambda$ represents the Cayley kernel. This quotient quandle can be embedded into $\conj(\aut(Q))$.

Matrix groups hold particular interest due to Kuperberg's certificate involving coloring by $\conj(\GL(2,q))$. His proof also suggests that the problem of coloring by $\conj(\GL(2,q))$ is challenging in general, prompting us to begin with a simpler class: the torus knots. By investigating coloring invariants in this specific context, we can gain valuable insights and potentially extend our understanding to more complex classes of knots.

Coloring of torus knots has garnered considerable attention and has been explored from various perspectives, as evidenced by the works \cite{Zhou}, \cite{Basi}, \cite{Iwakiri}, and \cite{Asami}. Notably, the concept of coloring by Alexander Quandles has been thoroughly examined in the work of Asami and Kuga (\cite{Asami}). Building upon these contributions, this research paper delves into the analysis of coloring torus knots using conjugation quandles.

The organization of this work is structured as follows. In Section \ref{sec:fundamentals}, we lay the groundwork by introducing fundamental notions related to torus knots and quandles, thereby providing a necessary foundation for subsequent discussions. In Section \ref{sec:abstract_groups}, we present general results pertaining to the conjugation quandle colorings of torus knots over abstract groups. This includes our main theorem (see Theorem \ref{thm:char_kmp}), which characterizes the colorability of torus knots, as well as exploring various properties of colorings.
Subsequently, in Section \ref{sec:particular_groups}, we apply the results obtained in Section \ref{sec:abstract_groups} to establish characterization theorems for specific groups, such as the matrix groups $\GL(2,q)$ and $\SL(2,q)$, the dihedral group, and the symmetric group. The selection of these particular groups is twofold in purpose: firstly, they provide a manageable context for formulating elementary statements and characterization theorems, although the proofs might be notably technical. Secondly, these groups have been under investigation in other contexts and papers, as evidenced by relevant references in the bibliography. 
Section \ref{sec:conclusions} concludes this work, by discussing and presenting some questions that remained open about using conjugation quandles to understand the colorability of torus knots.

\section{Fundamentals} \label{sec:fundamentals}

We begin by introducing and reviewing the mathematical tool necessary for developing the upcoming theory. Firstly, we provide the definition of a quandle, which is the algebraic structure playing a fundamental role in this paper. Subsequently, we proceed to the geometric counterpart, where we introduce and construct torus knots, along with their representations and colorings. Finally, we conclude by presenting some well-known knot-theoretical results that serve as the basis for making additional assumptions on the parameters. For a comprehensive understanding of Knot and Quandle Theory, we recommend consulting the books by Murasugi \cite{murasugi1996knot} and Elhamdadi \cite{elhamdadi2015quandles}.

\begin{definition}
    A \textbf{(right) quandle} is a binar $(Q,\rhd)$ satisfying the following axioms:
    \begin{enumerate}
        \item $\forall x,y,z\in Q:\ (x\rhd y)\rhd z = (x\rhd z)\rhd (y\rhd z)$.
        \item $\forall x,y\in Q\ \exists! z\in Q:\ z\rhd x=y$.
        \item $\forall x\in Q:\ x\rhd x = x$.
    \end{enumerate}
\end{definition}

Out of a given group, we can construct an important class of quandles.

\begin{definition}
    Let $G$ be a group. The \textbf{(right) conjugation quandle over $G$} is the quandle obtained by taking $Q=G$ as the underlying set and $x\rhd y=yxy\mo$ as the binary operation. We denote it by $\conj(G)$.
\end{definition}

The quandle operation exhibits nice features in conjugation quandles. In the following lemma, we introduce several properties that will be consistently utilized throughout the entire paper, without mentioning them explicitly.

\begin{lemma}
    Let $G$ be a group. In $\conj(G)$, for every $x,y,x_i,y_i\in G$ and $k\in\N$ we have
    \begin{enumerate}
        \item $(x_1x_2)\rhd y = (x_1\rhd y)(x_2\rhd y)$.
        \item $x^k\rhd y = (x\rhd y)^k$.
        \item $x\rhd(y_1y_2) = (x\rhd y_2)\rhd y_1$.
        \item $x\rhd y^k = (\ldots((x\rhd y)\rhd y)\rhd\ldots)\rhd y$ ($k$ times).
    \end{enumerate}
\end{lemma}

\begin{proof}
    All of them are straightforward computations.
\end{proof}

We now proceed to introduce some fundamental notions of Knot Theory. All the knots in this paper are intended to be oriented. 

\begin{definition}
    Let $K$ be a regular diagram of an oriented knot, or simply, a knot. We denote by $\arcs(K)$ the set of connected strands of (the diagram) $K$. By a \textbf{(positive) crossing} in $K$ we mean a triple $(x,y,z)\in\arcs(K)^3$ such that \emph{$x$ passes under $y$ producing $z$} (see Figure \ref{fig:crossing}). In this case, we denote $x\rhd y=z$ and we call it \textbf{crossing relation} of $K$.
\end{definition}

\begin{figure}[ht]
    \centering
    \includegraphics[width=.25\linewidth]{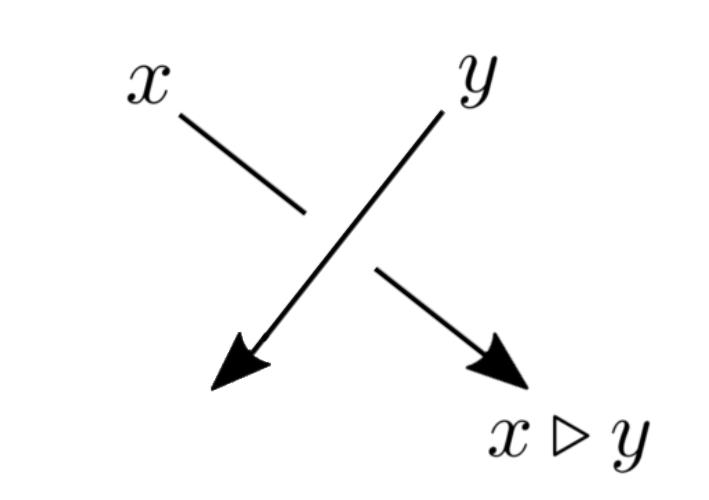}
    \caption{Crossing relation.}
    \label{fig:crossing}
\end{figure}

\begin{definition}
    Let $K$ be a knot and $(Q,\rhd)$ a quandle. A $Q$\textbf{-coloring} of $K$ is a mapping $c\colon\arcs(K)\to Q$ such that for every crossing $(x,y,z)$ of $K$ the equation $c(x)\rhd c(y)=c(z)$ holds in $Q$. If $c$ is the constant function, we say that it is the \textbf{trivial coloring}, every other coloring is called \textbf{non-trivial}. A knot $K$ is said to be \textbf{$Q$-colorable} if there exists a non-trivial coloring of $K$.
\end{definition}

\begin{remark}
    It is evident that every knot can be trivially colored; that is, a trivial coloring always exists. Therefore, our focus lies solely on the existence of non-trivial colorings. Consequently, we have chosen to define the $Q$-coloring in the manner described earlier, excluding trivial colorings.
\end{remark}

\begin{definition}
    Let $m,n$ be two positive integers and consider the braid with $n$ strands and $m$ twists as in Figure \ref{fig:kmn_std_diagram}. The $n$ leftmost strands are called \textbf{initial arcs}, the $n$ rightmost strands are called \textbf{terminal arcs}, and the $m$ diagonal strands are called \textbf{bridges}.
    The closure of such a braid, identifying each initial arc with the corresponding terminal arc, is called \textbf{$(m,n)$-torus knot (link)}, and denoted by $\K(m,n)$ (see Figure \ref{fig:k54_diagram}).
\end{definition}

\begin{figure}[ht] 
    \centering
    \includegraphics[width=.4\linewidth]{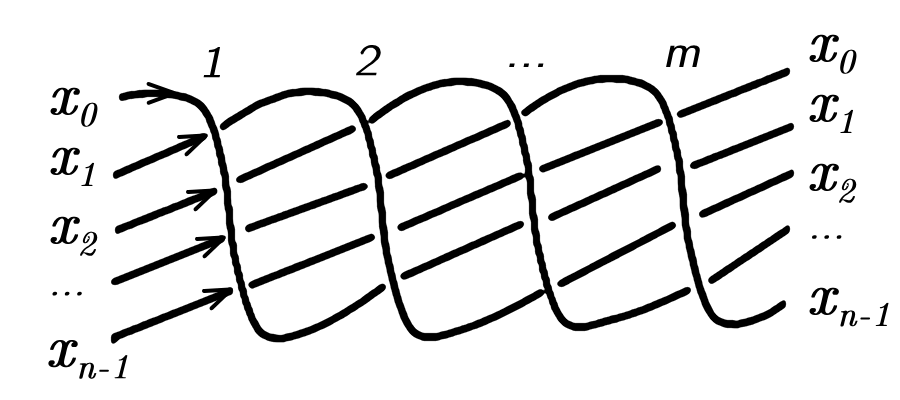}
    \caption{Braid diagram for $\K(m,n)$.}
    \label{fig:kmn_std_diagram}
\end{figure}

\begin{remark}
    Every torus knot (torus link) has the remarkable property of being embeddable on the surface of the trivial torus, without any points of self-intersection. Conversely, any knot lying on the surface of the trivial torus can be shown to be equivalent to $\K(m,n)$, for some integers $m$ and $n$. This explains the terminology used in Knot Theory. The depiction of this knot (link) on the trivial torus is shown in Figure \ref{fig:k54_diagram}. However, the diagram in Figure \ref{fig:k54_diagram} contains excessive information. Therefore, like many other authors, we prefer to use a more concise and schematic braid representation, as shown in Figure \ref{fig:kmn_std_diagram}. This simplified representation is known as the \emph{braid diagram} (or \emph{standard diagram}), and it disregards the specific identification presented in the knot definition. Certainly, when one is working with a torus knot, referring either to the initial arc or to the corresponding terminal arc naturally imparts the same information. When coloring is involved, we find the braid diagram particularly useful: it allows us to label the initial and terminal arcs with their respective colors easily. This labeling simplifies the coloring process and enhances our understanding of the knot's properties.
\end{remark}

\begin{figure}[ht]
    \centering
    \includegraphics[width=.4\linewidth]{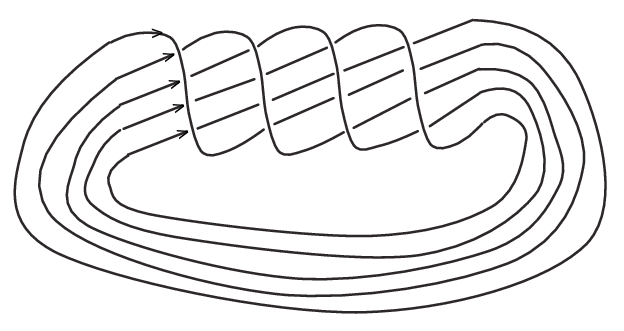}
    \caption{Knot diagram for $\K(5,4)$.}
    \label{fig:k54_diagram}
\end{figure}

There are two significant, well-known results that govern the overall behavior of a torus knot. Specifically, we have precise knowledge regarding how to set the parameters to achieve a pair of equivalent links or to cause the torus link to collapse into a knot (proofs and details can be found in Murasugi's book \cite{murasugi1996knot}).

\begin{theorem}[Classification of torus links] \label{thm:classification_torus_knots}
    Let $m,n,r,s\ge 2$. Then $\K(m,n)$ is equivalent to $\K(r,s)$ if and only if $\{m,n\}=\{r,s\}$.
\end{theorem}

\begin{theorem}[Torus knots and links] \label{thm:knot_link_gcd}
    Let $m,n\ge 1$. Then $\K(m,n)$ is a knot if and only if $\gcd(m,n)=1$.
\end{theorem}

\section{Conjugation quandle coloring of torus knots} \label{sec:abstract_groups}

\subsection{First characterization}

The first result is a characterization of coloring a torus knot using a conjugation quandle: coloring a knot with $\conj(G)$ is equivalent to finding a tuple of elements in $G$ that satisfy certain group term equations, or alternatively, certain quandle term equations.

\begin{definition}
    Let $m,n\in\N$ be such that $\gcd(m,n)=1$, let $G$ be a group, and let $x_0,\dots,x_{n-1}\in G$. We say that the tuple $(x_0,\dots,x_{n-1})$ \textbf{extends to a coloring} of $\K(m,n)$ if there exists a unique $\conj(G)$-coloring of $\K(m,n)$ of which $(x_0,\dots,x_{n-1})$ are the colors of the initial arcs, respectively. 
\end{definition}

\begin{remark}
    Given a coloring, we may always extract the tuple of the colors of the initial arcs. That tuple naturally extends to the given coloring. Moreover, observe that the constant tuple extends to the trivial coloring.
\end{remark}

\begin{theorem} \label{thm:characterization_1}
    Let $m,n\in\N$ be such that $\gcd(m,n)=1$, let $G$ be a group, and let $x_0,\dots,x_{n-1}\in G$ be not all equal. The following conditions are equivalent.
    \begin{enumerate}
        \item The tuple $(x_0,\dots,x_{n-1})$ extends to a (non-trivial) coloring for $\K(m,n)$.
        \item $|\{x_{0+i\pmod{n}}x_{1+i\pmod{n}}\ldots x_{m-1+i\pmod{n}}\colon i=0,\dots,n-1\}|=1.$ \label{thm:char2}
        \item For $u=\prod_{j=0}^{m-1}x_{n-m+j\pmod{n}}$, we have \label{thm:char3}
        $$x_i\rhd u = x_{i-m \pmod{n}}\quad\quad\forall i=0,\dots,n-1.$$ 
    \end{enumerate}
    Moreover, in this case, the element $u$ of \ref{thm:char3} is the common value in the set of \ref{thm:char2}.
\end{theorem}

\begin{proof} Prove the implications separately. All the indices of symbols $x$ are assumed to be computed modulo $n$ and all the indices of symbols $y$ are assumed to be computed modulo $m$.
    \begin{description}
        \item[$(i)\implies(iii)$] Denote by $y_0,\dots,y_{m-1}$ the colors of the bridges, as in the Figure \ref{fig:k54_xy}. 

        \begin{figure}[ht]
            \centering
            \includegraphics[width=.4\linewidth]{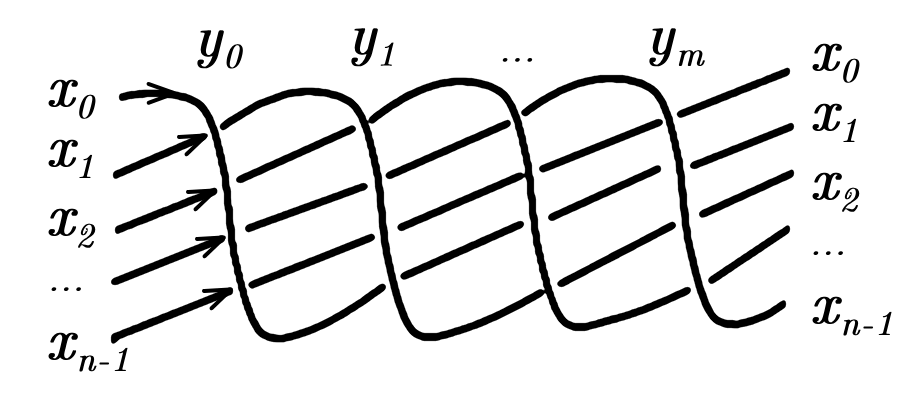}
            \caption{Colors in the proof of Theorem \ref{thm:characterization_1}}
            \label{fig:k54_xy}
        \end{figure}
        
        Observe that $x_0=y_0$. By the definition of coloring, we have the following relations
        \begin{equation} \label{eqn:yj_train_chop}
            (((y_j\rhd y_{j+1})\rhd y_{j+2})\rhd\ldots)\rhd y_{m-1}=x_{j+m-n}\quad\quad\forall j=0,\dots,m-1.
        \end{equation}
        Moreover, by the geometry of the torus knot, because of the $m$ twists, we also have
        \begin{equation} \label{eqn:xj_train}
            (((x_i\rhd y_0)\rhd y_1)\rhd\ldots)\rhd y_{m-1}=x_{i-m}\quad\quad\forall i=0,\dots,n-1.
        \end{equation}
        Now, if we expand equations (\ref{eqn:xj_train}), using the quandle axioms and equations (\ref{eqn:yj_train_chop}), we obtain
        \begin{align*}
            x_{i-m}&=(((x_i\rhd y_0)\rhd y_1)\rhd\ldots)\rhd y_{m-1} \\
            &=((((x_i\rhd y_1)\rhd y_2)\rhd\ldots)\rhd y_{m-1})\rhd ((((y_0\rhd y_1)\rhd y_2)\rhd\ldots)\rhd y_{m-1}) \\
            &=((((x_i\rhd y_1)\rhd y_2)\rhd\ldots)\rhd y_{m-1})\rhd x_{n-m} \\
            &=(((((x_i\rhd y_2)\rhd y_3)\rhd\ldots)\rhd y_{m-1})\rhd ((((y_1\rhd y_2)\rhd y_3)\rhd\ldots)\rhd y_{m-1})) \rhd x_{n-m} \\
            &=(((((x_i\rhd y_2)\rhd y_3)\rhd\ldots)\rhd y_{m-1})\rhd x_{n-m+1}) \rhd x_{n-m} \\
            &=\dots \\
            &=(((((x_i\rhd x_{n-1})\rhd x_{n-2})\rhd\ldots)\rhd x_{n-m+2})\rhd x_{n-m+1}) \rhd x_{n-m} \\
            &= (x_{n-m}x_{n-m+1}\ldots x_{n-2}x_{n-1})x_i(x_{n-1}\mo x_{n-2}\mo\ldots x_{n-m+1}\mo x_{n-m}\mo) \\
            &=x_i\rhd u.
        \end{align*}

        \item[$(iii)\implies(i)$] Associate the elements $x_0,\dots,x_{n-1}$ to the initial arcs, and compute the colors of the bridges $y_0,\dots,y_{m-1}$ recursively as follows:
        \begin{align*}
            y_0&=x_0 \\
            y_j&=(((x_j\rhd y_0)\rhd y_1)\rhd\ldots)\rhd y_{j-1}\quad\quad\forall j=0,\dots,m-1.
        \end{align*}
        This allows computing the color of all other elements in $\arcs(\K(m,n))$ and the conditions $x_i\rhd u = x_{i-m}$ guarantee the good definition of colors $x_i$'s for the diagram of $\K(m,n)$.
        
        \item[$(ii)\implies(iii)$] Note that \ref{thm:char2} can be seen as a chain of equations, where the terms are the product of the colors whose indices are consecutive and shifted by the same constant, possibly reduced modulo $n$. Fix $i\in\{0,\dots,n-1\}$. Then, because of \ref{thm:char2}, we have 
        \begin{align*}
            x_i\rhd u = ux_iu\mo &= (x_{n-m}x_{n-m+1}\ldots x_{n-2}x_{n-1})x_i(x_{n-m}x_{n-m+1}\ldots x_{n-2}x_{n-1})\mo \\
            &= (x_{n-m+i}x_{n-m+1+i}\ldots x_{n-2+i}x_{n-1+i})x_i(x_{n-m+i+1}x_{n-m+1+i+1}\ldots x_{n-2+i+1}x_{n-1+i+1})\mo \\
            &= (x_{i-m}x_{i-m+1}\ldots x_{i-2}x_{i-1})x_i(x_{i}\mo x_{i-1}\mo \ldots x_{i-m+2}\mo x_{i-m+1}\mo) = x_{i-m}. \\
        \end{align*}

        \item[$(iii)\implies(ii)$] Expand the equations as follows
        \begin{align} \label{eqn:char3_1}
            x_i\rhd u = x_{i-m} \iff (x_{n-m}x_{n-m+1}\ldots x_{n-2}x_{n-1})x_i=x_{i-m}(x_{n-m}x_{n-m+1}\ldots x_{n-2}x_{n-1}).
        \end{align}
        Set $i=0=n$ in (\ref{eqn:char3_1}), and cancel out the term $x_{n-m}$ on the left obtaining
        \begin{align} \label{eqn:char3_2}
            x_{n-m+1}x_{n-m+2}\ldots x_{n-2}x_{n-1}x_i=x_{n-m}x_{n-m+1}\ldots x_{n-2}x_{n-1},
        \end{align}
        which is one of the equations in \ref{thm:char2}. Substitute (\ref{eqn:char3_2}) in (\ref{eqn:char3_1}), set $i=n-m+1$, and cancel out the first term again to obtain another of the equations. Proceeding this way, since $\gcd(m,n)=1$ we obtain all the equations in $(ii)$. \qedhere
    \end{description}
\end{proof}

\begin{definition}
    Let $m,n\in\N$ be such that $\gcd(m,n)=1$, let $G$ be a group, and let $x_0,\dots,x_{n-1}\in G$. Let $(x_0,\dots,x_{n-1})$ extend to a coloring of $\K(m,n)$. We refer to the element $u=x_{n-m}x_{n-m+1}\ldots x_{n-2}x_{n-1}\in G$ (as in Theorem \ref{thm:characterization_1}) as the \textbf{harlequin} of $(x_0,\dots,x_{n-1})$.
\end{definition}

\subsection{Properties of conjugation quandle coloring}

It is natural to inquire whether a given knot coloring can be used to create a coloring for another knot. We observe that the answer to this question is frequently affirmative, and it involves certain divisibility conditions on the parameters of the torus knot. Throughout this subsection, $G$ denotes any fixed group.

\begin{proposition} \label{thm:mn_to_tmn}
     Let $m,n,t\in\N$ be such that $\gcd(m,n)=1$ and $\gcd(tm,n)=1$. If $\K(m,n)$ is $\conj(G)$-colorable, then also $\K(tm,n)$ is $\conj(G)$-colorable.
\end{proposition}

\begin{proof}
    The diagram of $\K(tm,n)$ can be obtained by gluing $t$ copies of the braid diagram of $\K(m,n)$, so we may use the given non-trivial coloring of $\K(m,n)$ to obtain a non-trivial coloring of $\K(tm,n)$.
\end{proof}

\begin{remark}
    In the previous proposition, the condition $\gcd(tm,n)=1$ is required only for $\K(tm,n)$ not to be a link (see Theorem \ref{thm:knot_link_gcd}), and it is not directly required in the proof. While many results presented here could be extended to links, we adhere to the convention of exclusively focusing on knots throughout this paper. This approach also involves assuming the greatest common condition divisor on the parameters.
\end{remark}

\begin{proposition} \label{thm:mn_to_mtn}
     Let $m,n,t\in\N$ be such that $\gcd(m,n)=1$ and $\gcd(m,tn)=1$. If $\K(m,n)$ is $\conj(G)$-colorable, then also $\K(m,tn)$ is $\conj(G)$-colorable.
\end{proposition}

\begin{proof}
    From Theorem \ref{thm:classification_torus_knots} and Proposition \ref{thm:mn_to_tmn}, we can infer that 
    \[
        \K(m,n)\mbox{ is colorable }\implies \K(n,m)\mbox{ is colorable }\implies \K(tn,m)\mbox{ is colorable }\implies \K(m,tn)\mbox{ is colorable}.\qedhere
    \] 
\end{proof}

\begin{proposition} \label{thm:mtn_to_mn}
     Let $m,n,t\in\N$ be such that $\gcd(m,n)=1$, and $\gcd(m,tn)=1$. Let $(y_0,\dots,y_{tn-1})$ extend to a coloring of $\K(m,tn)$, and define $$x_i=\prod_{j=0}^{t-1}y_{it+j},\quad\quad\mbox{for } i=0,\dots,n-1.$$
     Then $(x_0,\dots,x_{n-1})$ extends to a (possibly trivial) coloring of $\K(m,n)$.
\end{proposition}

\begin{proof}
    Let $v$ be the harlequin of $(y_0,\dots,y_{tn-1})$ in $\K(m,tn)$, and define $u=v^t$. We want to prove that $(x_0,\dots,x_{n-1})$ extends to a coloring of $\K(m,n)$ with harlequin $u$. Using the fact that $y_j\rhd v=y_{j-m\pmod{tn}}$, we have
    \[
        x_i\rhd u = \left(\prod_{j=0}^{t-1}y_{it+j}\right)\rhd v^t = \prod_{j=0}^{t-1}\left(y_{it+j}\rhd v^t\right) = \prod_{j=0}^{t-1} y_{it+j-tm\pmod{tn}} = \prod_{j=0}^{t-1} y_{(i-m)t+j\pmod{tn}} = x_{i-m\pmod{n}},
    \]
    therefore Theorem \ref{thm:characterization_1} applies.
\end{proof}

\begin{remark}
    The construction in the proof of Proposition \ref{thm:mtn_to_mn} has the possibility of producing a trivial coloring. However, this is not considered a drawback. In fact, we will utilize this feature in various proofs by contradiction in the subsequent discussions. We show this behavior in the following example.
\end{remark}

\begin{example}
    $\K(2,3)$ is $\conj(\S_3)$-colorable. In fact, $(x_0,x_1,x_2)$ extends to a coloring of $\K(2,3)$, where $x_0=(2\ 3), x_1=(1\ 2)$ and $x_2=(1\ 3)$. By Proposition \ref{thm:mn_to_mtn}, define
    \begin{align*}
        y_0 &=y_3=y_6=y_9=y_{12}=x_0=(2\ 3) \\
        y_1 &=y_4=y_7=y_{10}=y_{13}=x_1=(1\ 2) \\
        y_2 &=y_5=y_8=y_{11}=y_{14}=x_2=(1\ 3)
    \end{align*}
    and observe that $(y_0,\dots,y_{14})$ extends to a coloring of $\K(2,15)$.
    However, if we apply Proposition \ref{thm:mtn_to_mn} to the previous (non-trivial) coloring of $\K(2,15)$, we obtain a tuple which extends to a trivial coloring of $\K(2,5)$, indeed:
    \begin{align*}
        z_0 &= y_0y_1y_2 = (1\ 2) \\
        z_1 &= y_3y_4y_5 = (1\ 2) \\
        z_2 &= y_6y_7y_8 = (1\ 2) \\
        z_3 &= y_{9}y_{10}y_{11} = (1\ 2) \\
        z_4 &= y_{12}y_{13}y_{14} = (1\ 2).
    \end{align*}
    A direct computation with \textsf{GAP} shows, in fact, that $\K(2,5)$ is not $\conj(\S_3)$-colorable. 
\end{example}

\begin{lemma} \label{thm:j-i_trivial_coloring}
     Let $m,n\in\N$ be such that $\gcd(m,n)=1$, and let $(x_0,\dots,x_{n-1})$ extend to a coloring of $\K(m,n)$.  If $x_i=x_j$ for some colors $x_i,x_j$ with $\gcd(j-i,n)=1$, then $(x_0,\dots,x_{n-1})$ extends the trivial coloring. In particular, we have $x_0=x_1=\dots=x_{n-1}$.
\end{lemma}

\begin{proof}
    Since $\gcd(m,n)=\gcd(j-i,n)=1$, both $m$ and $j-i$ are invertible modulo $n$. Let $u$ be the harlequin of $(x_0,\dots,x_{n-1})$ and consider $k=m\mo(j-i)$. Then, for every $t\in\Z$, we have
    \[
        x_i=x_j \implies  x_i\rhd u^{tk} =x_j\rhd u^{tk} \implies x_{i-t(j-i)}= x_{j-t(j-i)}.
    \]
    In particular, $x_i=x_{i-t(j-i)}$ for all $t\in\Z$. Let $h\in\{0,\dots,n-1\}$ and define $\bar{t}=(j-i)\mo(i-h)$. Then $x_h=x_{i-\bar{t}(j-i)}$. Now, the choice of $h$ was arbitrary, therefore all the colors in the tuple are equal, hence, the coloring is trivial.
\end{proof}

\begin{corollary} \label{thm:trivial_or_distinct}
     Let $m\in\N$ and $p$ be a prime such that $p\nmid m$. Let $(x_0,\dots,x_{p-1})$ extend to a coloring of $\K(m,p)$. Then, either it extends to the trivial coloring, or the colors in the tuple $(x_0,\dots,x_{p-1})$ are all distinct. 
\end{corollary}

\begin{proof}
    Let $(x_0,\dots,x_{p-1})$ extend to a coloring of $\K(m,p)$, and assume that $x_i=x_j$ for some indices $0\le i<j\le p-1$. Since $p$ is prime, the condition $\gcd(j-i,p)=1$ trivially holds, so Lemma \ref{thm:j-i_trivial_coloring} applies.
\end{proof}

\subsection{The main characterization}

The importance of the divisors of the parameters becomes immediately apparent in light of the propositions presented in the preceding subsection. In this investigation, we shall first conclude our analysis of the relationship between these divisors and coloring. Our inquiry reveals that the colorability of a knot can be attributed to (at the very least) one of the prime divisors of the parameters, as demonstrated in Theorem \ref{thm:kmn_iff_kpq}. Following this, we proceed to establish a specific characterization (see Theorem \ref{thm:char_kmp}) concerning the coloring of $\K(m,n)$, which relies only on a single element in the group, in contrast to the dependence on $n$ elements as proven in Theorem $\ref{thm:characterization_1}$. 
This simplifies the process of verifying the colorability of a specific torus knot, owing to the involvement of fewer elements and the group-theoretical nature of the provided equivalent condition.
Throughout this subsection, $G$ denotes any fixed group. We start with the following arithmetic lemma.

\begin{lemma} \label{thm:gcd_lemma}
    Let $a,b\in\Z$ be such that $\gcd(a,b)=1$. Then $\gcd(a-b,ab)=1$.
\end{lemma}

\begin{proof}
    By contradiction, let $p$ be a prime such that $p\mid a-b$ and $p\mid ab$. Then, without loss of generality, $p\mid b$, which, together with the condition $p\mid a-b$, implies that $p\mid a$, against the assumption of $\gcd(a,b)=1$.
\end{proof}

\begin{proposition} \label{thm:kmpq_iff_kmp_or_kmq}
     Let $m,p,q\in\N$ be such that $\gcd(m,pq)=\gcd(p,q)=1$. Then, $\K(m,pq)$ is $\conj(G)$-colorable if and only if either $\K(m,p)$ or $\K(m,q)$ is $\conj(G)$-colorable.
\end{proposition}

\begin{proof}
    One implication follows directly from Proposition \ref{thm:mn_to_mtn}. Conversely, let $(x_0,\dots,x_{pq-1})$ extend to a non-trivial coloring of $\K(m,pq)$. By contradiction, assume that both $\K(m,p)$ and $\K(m,q)$ are not non-trivially colorable, and define
    $$y_i=\prod_{k=0}^{q-1}x_{iq+k},\quad\quad\mbox{for } i=0,\dots,p-1,$$
    $$z_j=\prod_{k=0}^{p-1}x_{jq+k},\quad\quad\mbox{for } j=0,\dots,q-1.$$
    By the assumption and Proposition \ref{thm:mtn_to_mn}, the tuples $(y_0,\dots,y_{p-1})$ and $(z_0,\dots,z_{q-1})$ extend to trivial colorings of $\K(m,p)$ and $\K(m,q)$, respectively, in particular this implies that $y_0=y_1=\dots=y_{p-1}$ and $z_0=z_1=\dots=z_{q-1}$, or equivalently
    \begin{equation} \label{eqn:k2p_trivial_coloring}
        x_0x_1\dots x_{q-1} = x_q x_{q+1}\dots x_{2q-1} = \dots = x_{(p-1)q}x_{(p-1)q+1}\dots x_{pq-1},
    \end{equation}
    \begin{equation} \label{eqn:k2q_trivial_coloring}
        x_0x_1\dots x_{p-1} = x_p x_{p+1}\dots x_{2p-1} = \dots = x_{(q-1)p}x_{(q-1)p+1}\dots x_{pq-1}.
    \end{equation}
    Proceed by case analysis, distinguishing among the possible values of $m$.
    
    \begin{description}
        \item[Case $m=2$:] From Theorem \ref{thm:characterization_1}, we have that
        \begin{equation} \label{eqn:k2pq_eqns}
            x_0x_1=x_1x_2=\dots=x_{pq-1}x_0.
        \end{equation}
        Thus, for every $i,j=0,\dots,p-1$, we have
        \begin{align*}
            y_i=y_j\ &\implies \prod_{k=0}^{q-1}x_{iq+k} = \prod_{k=0}^{q-1}x_{jq+k} \\
            &\implies 
            \begin{cases}
                \displaystyle{x_{iq+0}\prod_{k=1}^{q-1}x_{iq+k} = x_{jq+0}\prod_{k=1}^{q-1}x_{jq+k}} \\
                \displaystyle{\prod_{k=0}^{q-2}x_{iq+k} x_{(i+1)-1} = \prod_{k=0}^{q-2}x_{jq+k} x_{(j+1)-1}}
            \end{cases}
            \implies
            \begin{cases}
                x_{iq} = x_{jq} \\
                x_{(i+1)-1} = x_{(j+1)-1} 
            \end{cases}
        \end{align*}
        where we have canceled the product out because it is made by an even number of factors with consecutive indices, exploiting the equations of Theorem \ref{thm:characterization_1}\ref{thm:char2}, namely equations (\ref{eqn:k2pq_eqns}). Now, deleting the first and the last term of each of the equations (\ref{eqn:k2p_trivial_coloring}) and iterating this procedure, we obtain that
        \begin{equation} \label{eqn:k2pq_partial1}
            x_{iq+k} = x_{jq+k}\quad\quad\forall i,j=0,\dots,p-1,\,\forall k=0,\dots,q-1.
        \end{equation}
        Proceeding in the same way, using $(z_0,\dots,z_{q-1})$, we get
        \begin{equation} \label{eqn:k2pq_partial2}
            x_{sp+h} = x_{tp+h}\quad\quad\forall s,t=0,\dots,q-1,\,\forall h=0,\dots,p-1.
        \end{equation}
        We may rewrite equations (\ref{eqn:k2pq_partial1}) and (\ref{eqn:k2pq_partial2}) in a more compact and equivalent form
        \begin{align} 
            x_{iq+k} &= x_{k}\quad\quad\forall i=0,\dots,p-1,\,\forall k=0,\dots,q-1, \label{eqn:k2pq_partial3}\\
            x_{jq+h} &= x_{h}\quad\quad\forall j=0,\dots,q-1,\,\forall h=0,\dots,p-1. \label{eqn:k2pq_partial4}
        \end{align}
        Now, consider any $c\in\{0,\dots pq-1\}$. By the conditions on $m,p$ and $q$, we may assume that $2<p<q$. By the division with reminder theorem we have $c=aq+t, t=bp+r$ for some unique $a,b,t,r\in\mathbb{Z}$ with $0\le t<q$ and $0\le r \le p-1 <q-1$, hence $c=aq+bp+r$. Consider the Diophantine equation $jp+iq=r$, which has solution $(i,j)\in\mathbb{Z}^2$ because $1=\gcd(p,q)\mid r$. Then we have $c=(b+j)p+(a+i)q$, and, computing indices modulo $pq$ together with equations (\ref{eqn:k2pq_partial3}) and (\ref{eqn:k2pq_partial4}), we have
        $$x_c=x_{(b+j)p+(a+i)q} = x_{(b+j)p+(a+i)q+0 } = x_{(a+i)q+0} = x_0.$$
        Now, the choice of $c$ was arbitrary, therefore all the colors $x_0,\dots,x_{pq-1}$ are equal, which is a contradiction.
        
        \item[Case $m>2$:] Since $\gcd(q,m)=1$, consider $t=q\mo\pmod{m}$. By the equations (\ref{eqn:k2p_trivial_coloring}), multiplying $t$ elements with consecutive indices (possibly with repetitions), we have
        \begin{align*}
            y_0y_1\dots y_{t-1}=y_1y_2\dots y_{t} \implies (x_0x_1\dots x_{q-1})y_1\dots y_{t-1} = (x_{q}x_{q+1}\dots x_{2q-1})y_2\dots y_{t} \implies x_0 = x_{q},
        \end{align*}
        where the last implication holds because the products are made by $tq-1$ factors with consecutive indices, which is divisible by $m$, thus Theorem \ref{thm:characterization_1}\ref{thm:char2} holds.
        Proceeding in the same way, using $(z_0,\dots,z_{q-1})$, we obtain that $x_0 = x_{p}$. Thus, we have $x_{p}=x_{q}$ and $\gcd(p,q)=1$ by assumption. Because of Lemma \ref{thm:gcd_lemma} we also get $\gcd(p-q,pq)=1$, which implies that all the colors are equal because of Lemma \ref{thm:j-i_trivial_coloring}, and this is a contradiction. \qedhere
    \end{description}
\end{proof}

\begin{proposition} \label{thm:kmpe_iff_kmp}
     Let $m,p,e\in\N$ such that $p$ is prime and $p\nmid m$. Then, $\K(m,p^e)$ is $\conj(G)$-colorable if and only if $\K(m,p)$ is $\conj(G)$-colorable.
\end{proposition}

\begin{proof}
    One implication follows directly from Proposition \ref{thm:mn_to_mtn}. Conversely, let $(x_0,\dots, x_{p^e-1})$ extend to a non-trivial coloring of $\K(m,p^e)$. There are three possible cases:
    \begin{enumerate}
        \item[(a)] $x_i=x_j$ if and only if $i\equiv j\pmod{p}$.
        \item[(b)] There are $i,j\in\{0,\dots,p^{e-1}-1\}$ such that $x_i=x_j$ if and only if $i\not\equiv j\pmod{p}$.
        \item[(c)] All the $x_i$'s are different colors.
    \end{enumerate}
    Observe that case (a) provides a non-trivial coloring for $\K(m,p)$, because the equations in Theorem \ref{thm:characterization_1} hold already. Case (b) is impossible because it would contradict Lemma \ref{thm:j-i_trivial_coloring}, being $\gcd(j-i,p^e)=1$. Assume case (c). Let $u$ be the harlequin of $(x_0,\dots, x_{p^e-1})$, and define
    $$y_i=\prod_{k=0}^{p^{e-1}-1}x_{ip^{e-1}+k},\quad\quad\mbox{for } i=0,\dots,p-1.$$
    From Proposition \ref{thm:mtn_to_mn} we know that it extends to a coloring of $\K(m,p)$ with harlequin $v=u^{p^{e-1}}$. Proceed by case analysis, distinguishing among the possible values of $m$.
    
    \begin{description}
        \item[Case $m=2$:] If the colors $y_0,\dots,y_{p-1}$ are all different, the proof is completed. Assume that two of them are equal, say $y_i=y_j$ for some distinct $i,j\in\{0,\dots,p-1\}$. Then
        \begin{align*}
            y_i=y_j &\implies \prod_{k=0}^{p^{e-1}-1}x_{ip^{e-1}+k} = \prod_{k=0}^{p^{e-1}-1}x_{jp^{e-1}+k} \\ 
            &\implies x_{ip^{e-1}}\prod_{k=1}^{p^{e-1}-1}x_{ip^{e-1}+k} = x_{jp^{e-1}}\prod_{k=1}^{p^{e-1}-1}x_{jp^{e-1}+k} \implies x_{ip^{e-1}} = x_{jp^{e-1}}
        \end{align*}
        where we have canceled the product out because it is made by an even number of factors with consecutive indices, exploiting the equations of Theorem \ref{thm:characterization_1}\ref{thm:char2}. This leads to a contradiction because we are assuming that all the colors are different, hence this sub-case is indeed impossible. 
        
        \item[Case $m>2$:] From Corollary \ref{thm:trivial_or_distinct} we know that either the elements of the tuple $(y_0,\dots,y_{p-1})$ are all distinct, or $(y_0,\dots,y_{p-1})$ extends to the trivial coloring. In the first case, the proof is completed. Assume that $y_0=y_1=\dots=y_{p-1}$. Since $\gcd(m,p)=1$, consider $t=p\mo\pmod{m}$. Multiplying $t$ elements with consecutive indices (possibly with repetitions), we have
        \begin{align*}
            y_0y_1\dots y_{t-1}=y_1y_2\dots y_{t} \implies (x_0x_1\dots x_{p-1})y_1\dots y_{t-1} = (x_{p}x_{p+1}\dots x_{2p-1})y_2\dots y_{t} \implies x_0 = x_{p},
        \end{align*}
        where the last implication holds because the products are made by $tp-1$ factors, which is divisible by $m$, thus Theorem \ref{thm:characterization_1}\ref{thm:char2} holds. This leads to a contradiction because we are assuming that all the colors of $(x_0,\dots,x_{p^e-1})$ are different, hence also this sub-case is impossible. \qedhere
    \end{description}
\end{proof}

\begin{proposition} \label{thm:kmn_iff_kmq} 
     Let $m,n\in\N$ be such that $\gcd(m,n)=1$. Then, $\K(m,n)$ is $\conj(G)$-colorable if and only if there is a prime factor $q$ of $n$ such that $\K(m,q)$ is $\conj(G)$-colorable.
\end{proposition}

\begin{proof}
    One direction follows from Proposition \ref{thm:mn_to_mtn}. For the other, write the prime factorization of $n=\prod_{i=1}^k p_i^{e_i}$ and conclude with a simple  induction argument using Propositions  \ref{thm:kmpq_iff_kmp_or_kmq} and \ref{thm:kmpe_iff_kmp}.
\end{proof}

\begin{theorem} \label{thm:kmn_iff_kpq}
    Let $m,n\in\N$ be such that $\gcd(m,n)=1$. Then, $\K(m,n)$ is $\conj(G)$-colorable if and only if there is a prime factor $p$ of $m$ and a prime factor $q$ of $n$ such that $\K(p,q)$ is $\conj(G)$-colorable.
\end{theorem}

\begin{proof}
    One direction follows from Propositions \ref{thm:mn_to_tmn} and \ref{thm:mn_to_mtn}. For the converse, we can use Proposition \ref{thm:kmn_iff_kmq}, together with Theorem \ref{thm:classification_torus_knots}, to 
    infer that
    \begin{align*}
        \K(m,n)\mbox{ is colorable }&\implies
        \K(m,q)\mbox{ is colorable for some } q\mid n \\
        &\implies\K(q,m)\mbox{ is colorable for some } q\mid n \\
        &\implies\K(q,p)\mbox{ is colorable for some } q\mid n\mbox{ and } p\mid m \\
        &\implies\K(p,q)\mbox{ is colorable for some } p\mid m\mbox{ and } q\mid n.\qedhere
    \end{align*}
\end{proof}

\begin{remark} \label{rmk:parameters_assumptions}
    Due to Theorem \ref{thm:kmn_iff_kpq}, when examining the colorability of $\K(m,n)$ we can assume that $n$ is prime and that $m\nmid n$. We will see that assuming $m$ to be prime as well is inconsequential. It is important to note that if $m=1$ then $\K(1,n)$ is only trivially colorable, so we can also assume $m\neq 1$. 
\end{remark}

We present now the main results of this paper, related to the study of conjugation quandle coloring of torus knots. 

\begin{theorem} \label{thm:char_kmp}
    Let $G$ be a group, and let $m,p\in\N$ be such that $m\ge 2$ and $p$ is prime with $p\nmid m$. Then $\K(m,p)$ is $\conj(G)$-colorable if and only if there is $u\in G$ such that the centralizers $\C_G(u^p)\setminus\C_G(u)\neq\emptyset$.
\end{theorem}

\begin{proof}
    We prove that $\K(m,p)$ is $\conj(G)$-colorable if and only if there are $x_0,u\in G$ such that 
    \[
    \begin{cases}
        ux_0u\mo\neq x_0 \\
        u^px_0u^{-p}=x_0,
    \end{cases}
    \]
    which is equivalent to the condition on the centralizers in the statement. 

    Let $(x_0,\dots x_{p-1})$ extend to a non-trivial coloring of $\K(m,p)$ with harlequin $u$. Because of Corollary \ref{thm:trivial_or_distinct}, all the colors must be distinct. Because of Theorem \ref{thm:characterization_1}, for every $t\in\{0,\dots,p-1\}$ we have $x_0\rhd u^k=x_t$ for $k=-m\mo t\pmod{p}$, that is, all the colors can be obtained from $x_0$ and (a suitable power of) $u$. Therefore, we have $x_0,u\in G$ such that
    \[
    \begin{cases}
        ux_0u\mo=x_{-m} \\
        ux_{-m}u\mo=x_{-2m} \\
        \vdots \\
        ux_{-(p-1)}u\mo=x_{-p}=x_0
    \end{cases}
    \]
    which is equivalent to 
    \[
    \begin{cases}
        ux_0u\mo=x_{-m}\neq x_0 \\
        u^px_0u^{-p}=x_0 \\
    \end{cases}
    \]
    where the second equation is obtained by combining all the previous ones. Note that, since all the colors are different, in this setting, we do not need to require also $u^ix_0u^{-i}\neq x_0$ for all $i=0,\dots,p-1$, because if we had both $u^i,u^p\in\C_G(x_0)$, being the centralizer a subgroup of $G$, we would also have $u^{\gcd(i,p)}=u^1=u\in\C_G(x_0)$, which is forbidden by the first equation.
    Conversely, let $x_0,u\in G$ as in the statement, and define $x_{-im}=u^ix_0u^{-i}$ for $i=\{0,\dots,p-1\}$. This is indeed a coloring because, by definition of conjugation quandle, we have $x_{-im}\rhd u=x_{-im-m}$, and $x_{-pm}=x_0$. Moreover, since $\gcd(m,p)=1$, this is enough to define all the colors.
\end{proof}

\begin{remark} \label{rmk:m_irrelevant}
    Observe that the colorability of $\K(m,p)$ is independent on $m$. This allows us to designate the initial entry as any convenient number, provided that the second parameter is assumed to be prime.
\end{remark}

\begin{remark} \label{rmk:fast_trick}
    In the notation of Theorem \ref{thm:char_kmp}, note that if there is an element $u\in G$ of order $p$ such that $u\not\in\mathsf{Z}(G)$, then $\K(m,p)$ is $\conj(G)$-colorable.
\end{remark}

\section{Coloring with particular groups} \label{sec:particular_groups}

We now apply the theorems mentioned above to specific cases. We initiate the analysis by considering matrix groups, and then proceed with dihedral and symmetric groups.

\subsection[General linear group]{General linear group} \label{ssc:GL}

In this subsection, with $p$ and $q$ we denote two prime numbers, and the group $\GL(2,q)$ is denoted by $G$. Our objective is to derive a numerical characterization for the colorability of $\K(m,p)$ solely in terms of $p$ and $q$, as we have determined that the parameter $m$ is irrelevant (see Remark \ref{rmk:m_irrelevant}).

\begin{remark}
    The following table displays the representatives of the conjugacy classes of $\GL(2,q)$, together with their centralizers. In virtue of Theorem \ref{thm:char_kmp}, for a representative $u$ of each conjugacy class, we want to compute when the condition $\C_{\GL(2,q)}(u^p)\setminus\C_{\GL(2,q)}(u)\neq\emptyset$ holds, so we aim to recreate the table with $u^p$ instead of $u$ and compare the results.
\end{remark}

\begin{center}
    \begin{tabular}{ c|l|l }
    Type & $u$ & $\C_{\GL(2,q)}(u)$ \\
    \hline
    Type 1 &
    $\begin{pmatrix}
        a & 0 \\
        0 & a
    \end{pmatrix}\quad a\neq 0$ & 
    $\GL(2,q)$ \Tstrut \Bstrut\\[15pt]
    
    Type 2 & 
    $\begin{pmatrix}
        a & 0 \\
        0 & b
    \end{pmatrix}\quad 0<a<b$ & 
    $\left\{\begin{pmatrix}
        u & 0 \\
        0 & v
    \end{pmatrix}\in\GL(2,q)\colon u,v\neq 0\right\}$ \\[15pt]

    Type 3 & 
    $\begin{pmatrix}
        a & 1 \\
        0 & a
    \end{pmatrix}\quad a\neq 0$ & 
    $\left\{\begin{pmatrix}
        u & v \\
        0 & u
    \end{pmatrix}\in\GL(2,q)\colon u\neq 0\right\}$ \\[15pt] 

    Type 4 & 
    $\begin{pmatrix}
        0 & 1 \\
        a & b
    \end{pmatrix}\quad x^2-bx-a\mbox{ irreducible}$ & 
    $\left\{\begin{pmatrix}
        u & v \\
        au & u+bv
    \end{pmatrix}\in\GL(2,q)\colon u\neq 0\mbox{ or } v\neq 0\right\}$ \\  
    \end{tabular}
\end{center}

\begin{proposition} \label{prop:type1}
    Let $u\in G$ be a matrix of type 1. Then $\C_G(u)=\C_G(u^p)=G$.
\end{proposition}

\begin{proof}
    The matrix power $u^p$ is still a scalar matrix, hence its centralizer is again maximal.
\end{proof}

\begin{proposition} \label{prop:type2}
    Let $u\in G$ be a matrix of type 2. Then  $\C_G(u^p)\setminus\C_G(u)\neq\emptyset$ if and only if $p\mid q-1$.
\end{proposition}

\begin{proof}
    The matrix power $u^p=\big(\begin{smallmatrix}
    a^p & 0\\
    0 & b^p
    \end{smallmatrix}\big)$ is still diagonal, so its centralizer is strictly larger if and only if $u^p$ is a scalar matrix, that is when $a^p\equiv b^p\pmod{q}$. This happens if and only if $(ab\mo)^p\equiv 1\pmod{q}$ that is when $\mathbb{F}_q^\times$ has an element $u$ of order $p$ of the form $u=ab\mo$, or equivalently when $p\mid q-1$, by Cauchy's Theorem.
\end{proof}

\begin{proposition} \label{prop:type3}
    Let $u\in G$ be a matrix of type 3. Then  $\C_G(u^p)\setminus\C_G(u)\neq\emptyset$ if and only if $p=q$.
\end{proposition}

\begin{proof}
    A direct computation shows that the matrix power $$u^p=\begin{pmatrix}
        a & 1 \\
        0 & a
    \end{pmatrix}^p = 
    \begin{pmatrix}
        a^p & pa^{p-1} \\
        0 & a^p
    \end{pmatrix}$$
    has the same centralizer as the matrix $u$, unless $u^p$ is a scalar matrix. This happens when $pa^{p-1}\equiv 0\pmod{q}$ that is, when $p=q$, being both primes.
\end{proof}

\begin{lemma} \label{lem:type4_xnyn}
    Let $u=\big(\begin{smallmatrix}
    0 & 1\\
    a & b
    \end{smallmatrix}\big)\in G$ be a matrix of type 4. Then for every $n\ge 1$ we have
    \[
        \begin{pmatrix}
        0 & 1 \\
        a & b
        \end{pmatrix}^{n}=        
        \begin{pmatrix}
        x_{n-1} & y_{n-1} \\
        x_{n} & y_{n}
    \end{pmatrix}
    \]
    where
    \[
    \begin{cases}
        x_{0} = 0 \\
        y_{0} = 1
    \end{cases},
    \quad\quad
    \begin{cases}
        x_{n} = ay_{n-1} \\
        y_{n} = x_{n-1} + by_{n-1}.
    \end{cases}\quad n\ge 1.
    \]
\end{lemma}

\begin{proof}
    It follows easily by induction.
\end{proof}

\begin{lemma} \label{lem:type4_yn}
    In the notation of Lemma \ref{lem:type4_xnyn}, assuming $q\neq 2$, we have
    \[
    y_n=\frac{d\mo}{2}\left((d+c)^{n+1}+(d-c)^{n+1}\right)
    \]
    where $c=\frac{b}{2}$ and $d=\frac{\sqrt{b^2+4a}}{2}$.
\end{lemma}

\begin{proof}
    By Lemma \ref{lem:type4_xnyn} we have for every $n\in\mathbb{N}$
    \[
    \begin{pmatrix}
    x_n \\
    y_n
    \end{pmatrix}
    =
    \begin{pmatrix}
    0 & a \\
    1 & b
    \end{pmatrix}^n
    \begin{pmatrix}
    0 \\
    1
    \end{pmatrix}.
    \]
    Let $A=
    \big(\begin{smallmatrix}
    0 & a \\
    1 & b
    \end{smallmatrix}\big)$.    
    Compute the matrix power $A^n$ using the diagonalization technique. Let $\lambda_1=c+d,\lambda_2=c-d\in\mathbb{F}_{q^2}$ be the eigenvalues of $A$, and let $U\in\GL(2,q^2)$ be the matrix such that $U\mo AU=
    \big(\begin{smallmatrix}
    \lambda_1 & 0\\
    0 & \lambda_2
    \end{smallmatrix}\big)$. Hence $A^n=
    U\big(\begin{smallmatrix}
    \lambda_1^n & 0\\
    0 & \lambda_2^n
    \end{smallmatrix}\big)U\mo$ and 
    \[
    \begin{pmatrix}
    x_n \\
    y_n
    \end{pmatrix}
    =
    \begin{pmatrix}
    k & l \\
    m & r
    \end{pmatrix}
    \begin{pmatrix}
    \lambda_1^n \\
    \lambda_2^n
    \end{pmatrix}
    \]
    for some $k,l,m,r\in\mathbb{F}_{q^2}$. Using the known conditions $x_0=0,y_0=1,x_1=a,y_1=b$ we obtain 
    \[
    k=\frac{a}{\lambda_1-\lambda_2}=\frac{a}{2d},\quad\quad
    l=\frac{-a}{\lambda_1-\lambda_2}=\frac{-a}{2d},\quad\quad
    m=\frac{b-\lambda_2}{\lambda_1-\lambda_2}=\frac{d+c}{2d},\quad\quad
    r=\frac{\lambda_1-b}{\lambda_1-\lambda_2}=\frac{d-c}{2d}.
    \]
    Now compute the powers $\lambda_1^n$ and $\lambda_2^n$ using the Binomial Theorem:
    \begin{align*}
        \lambda_1^n &=(c+d)^n=\sum_{k=0}^{\left\lfloor\frac{n}{2}\right\rfloor}\binom{n}{2k}d^{2k}c^{n-2k} + \sum_{k=0}^{\left\lfloor\frac{n}{2}\right\rfloor}\binom{n}{2k+1}d^{2k+1}c^{n-2k-1} \\
        \lambda_2^n &=(c-d)^n=\sum_{k=0}^{\left\lfloor\frac{n}{2}\right\rfloor}\binom{n}{2k}d^{2k}c^{n-2k} - \sum_{k=0}^{\left\lfloor\frac{n}{2}\right\rfloor}\binom{n}{2k+1}d^{2k+1}c^{n-2k-1}
    \end{align*}
    Therefore
    \[
        y_n=m\lambda_1^n+r\lambda_2^n =\ \dots\ = d\mo\sum_{\substack{i=0 \\ i\, odd}}^{n+1}\binom{n+1}{i}d^ic^{(n+1)-i} = \frac{d\mo}{2}\left((d+c)^{n+1}+(d-c)^{n+1}\right).
    \]
\end{proof}

\begin{lemma} \label{lem:type4_elem_ord_p}
    The condition $p\mid q+1$ holds if and only if there are $p-1$ elements $u$ of multiplicative order $p$ in $\mathbb{F}_{q^2}$, all belonging to $\mathbb{F}_{q^2}\setminus\mathbb{F}_{q}$, and each of those can be expressed as
    \[
        u=\frac{c+d}{c-d}
    \]
    where $c\in\mathbb{F}_{q}$ and $d\in\mathbb{F}_{q^2}\setminus\mathbb{F}_{q}$.
\end{lemma}

\begin{proof}
    Assume that $p\mid q+1$, thus $p\mid q^2-1=|\mathbb{F}_{q^2}|$. Under this assumption, it is known that in $\mathbb{F}_{q^2}$ there are exactly $p-1>0$ elements of order $p$. Assume, by contradiction, that one of those was in $\mathbb{F}_{q}$. Then it would generate a multiplicative cyclic subgroup of $\mathbb{F}_{q}^\times$ of order $p$ containing all such elements, and implying that $p\mid q-1$ which is impossible, being $p\neq 2$. Then every element of order $p$ must belong to $\mathbb{F}_{q^2}\setminus\mathbb{F}_{q}$.    Consider now, for every $c\in\mathbb{F}_{q^2}^\times$ the map $\varphi_c\colon\mathbb{F}_{q^2}\to\mathbb{F}_{q^2}$ defined by
    \[
    \varphi_c(x)=
    \begin{cases}
        \displaystyle{\frac{c+x}{c-x}}
        \quad x\neq c \\
        -1\quad\quad x=c
    \end{cases}.
    \]
    It is easy to see that $\varphi_c$ is bijective. Since $p>2$, the element $-1$ does not have order $p$. Therefore all the $p-1$ elements of order $p$ are contained in $\varphi_c\left(\mathbb{F}_{q^2}\setminus\{c\}\right)$. Note that, if for an element $u\in\mathbb{F}_{q^2}\setminus\mathbb{F}_{q}$ of order $p$ we have $u=\varphi_c(x)$, for some $c\in\mathbb{F}_{q}$, then, necessarily, we need to have $x\in\mathbb{F}_{q^2}\setminus\mathbb{F}_{q}$, otherwise, we would get that $u\in\mathbb{F}_{q}$.
    
    Conversely, the existence of elements of order $p$ in $u\in\mathbb{F}_{q^2}\setminus\mathbb{F}_{q}$ implies that $p\mid q^2-1=|\mathbb{F}_{q^2}^\times|$. If we had $p\mid q-1=|\mathbb{F}_{q}^\times|$ then by Cauchy's theorem, $\mathbb{F}_{q}^\times$ would contain one (hence, all) of them, which is against the assumptions. It follows that $p\mid q+1$.
\end{proof}

\begin{lemma} \label{lem:type4_ypm1zero}
    In the notation of Lemma \ref{lem:type4_xnyn} and Lemma \ref{lem:type4_yn}, we have $y_{p-1}\equiv 0\pmod{q}$ if and only if $p\mid q+1$.
\end{lemma}

\begin{proof}
    Analyze first when $q=2$. The only case in which the polynomial $x^2-bx-a$ is irreducible in $\mathbb{F}_2$ is when $a=b=1$, and its splitting field is $\mathbb{F}_2[x]/\langle x^2+x+1\rangle\cong\mathbb{F}_4$. Assuming $a=b=1$, a simple induction argument shows that $y_n\equiv 0\pmod{2}$ if and only if $n\equiv 2\pmod{3}$, thus, $y_{p-1}\equiv{0}$ if and only if $p\equiv 0\pmod{3}$, that is $p=3$. Therefore, the claim holds for $q=2$.

    Assume now $q\neq 2$, then Lemma \ref{lem:type4_yn} applies and we have 
    \begin{align*}
        y_{p-1}\equiv 0\pmod{q} \iff \frac{d\mo}{2}\left((d+c)^{n+1}+(d-c)^{n+1}\right)\equiv 0\pmod{q} \iff \left((c+d)(c-d)\mo\right)^p\equiv 1\pmod{q}
    \end{align*}
    which holds if and only if the equation $u^p=1$ has solution in $\mathbb{F}_{q^2}$ for some $u$ of the form $u=\frac{c+d}{c-d}$ with $c\in\mathbb{F}_{q}$ and $d\in\mathbb{F}_{q^2}\setminus\mathbb{F}_{q}$. The conclusion now follows from Lemma \ref{lem:type4_elem_ord_p}.
\end{proof}

\begin{proposition} \label{prop:type4}
    Let $u\in G$ be a matrix of type 4. Then  $\C_G(u^p)\setminus\C_G(u)\neq\emptyset$ if and only if $p\mid q+1$.
\end{proposition}

\begin{proof}
    A direct computation shows that the matrix power $u^p$ has the same centralizer as the matrix $u$ unless $u^p$ is a scalar matrix, and this happens if and only if $y_{p-1}\equiv 0\pmod{q}$. The conclusion now follows from Lemma \ref{lem:type4_ypm1zero}.
\end{proof}

We summarize Propositions \ref{prop:type1}, \ref{prop:type2}, \ref{prop:type3}, \ref{prop:type4} in the following table.

\begin{center}
    \begin{tabular}{ c|l|c }
    Type & $u^p$ & $\C_{\GL(2,q)}(u^p)\setminus\C_{\GL(2,q)}(u)\neq\emptyset$ \\
    \hline
    Type 1 &
    $\begin{pmatrix}
        a^p & 0 \\
        0 & a^p
    \end{pmatrix}\quad a\neq 0$ & 
    Never \Tstrut \Bstrut\\[15pt]
    
    Type 2 & 
    $\begin{pmatrix}
        a^p & 0 \\
        0 & b^p
    \end{pmatrix}\quad 0<a<b$ & 
    $p\mid q-1$ \\[15pt]

    Type 3 & 
    $\begin{pmatrix}
        a^p & pa^{p-1} \\
        0 & a^p
    \end{pmatrix}\quad a\neq 0$ & 
    $p=q$ \\[15pt] 

    Type 4 & 
    $\begin{pmatrix}
        x_{p-1} & y_{p-1} \\
        ay_{p-1} & x_{p-1}+by_{p-1}
    \end{pmatrix}\quad x^2-bx-a\mbox{ irreducible}$ & 
    $p\mid q+1$\\  
    \end{tabular}
\end{center}

In conclusion, by combining all the information acquired from the previous results with the fact that conjugacy classes form a partition of $G$, we obtain the main result of this subsection.

\begin{theorem}[$\conj(\GL(2,q))$-coloring of Torus Knots] \label{thm:GL_characterization_coloring}
    Let $m,p\in\N$ be such that $m\ge 2$ and $p$ is prime with $p\nmid m$. The torus knot $\K(m,p)$ is $\conj(\GL(2,q))$-colorable if and only if $p\mid q(q+1)(q-1)$.
\end{theorem}

\subsection[Special linear group]{Special linear group}

We proceed in a similar manner as in the case of $\GL(2,q)$, distinguishing among representatives of conjugacy classes and their centralizers. Once again, throughout this subsection, with $p$ and $q$ we denote two prime numbers, and the special linear group $\SL(2,q)$ is denoted by $G$. Our objective is to obtain a numerical characterization for the colorability of $\K(m,p)$ solely in terms of $p$ and $q$.

\begin{center}
    \begin{tabular}{ c|l|l }
    Type & $u$ & $\C_{\SL(2,q)}(u)$ \\
    \hline
    Type 1 &
    $\begin{pmatrix}
        a & 0 \\
        0 & a
    \end{pmatrix}\quad a^2=1$ & 
    $\SL(2,q)$ \Tstrut \Bstrut\\[15pt]
    
    Type 2 & 
    $\begin{pmatrix}
        a & 0 \\
        0 & a\mo
    \end{pmatrix}\quad a\neq 0$ & 
    $\left\{\begin{pmatrix}
        u & 0 \\
        0 & u\mo
    \end{pmatrix}\in\SL(2,q)\colon u\neq 0\right\}$ \\[15pt]

    Type 3 & 
    $\begin{pmatrix}
        a & b \\
        0 & a
    \end{pmatrix}\quad a^2=1, b=1\mbox{ or } b\mbox{ non-square}$ & 
    $\left\{\begin{pmatrix}
        u & v \\
        0 & u
    \end{pmatrix}\in\SL(2,q)\colon u^2=1\right\}$ \\[15pt] 

    Type 4 & 
    $\begin{pmatrix}
        0 & 1 \\
        -1 & a
    \end{pmatrix}\quad a=r+r^q, r\in\mathbb{F}_{q^2}\setminus\mathbb{F}_{q}, r^{q+1}=1$ & 
    $\left\{\begin{pmatrix}
        u & v \\
        -v & u+av
    \end{pmatrix}\in\SL(2,q)\colon u(u+av)+v^2=1\right\}$ \\  
    \end{tabular}
\end{center}

\begin{remark}
    The subsequent propositions can be proven using the same techniques presented for the case of $\GL(2,q)$ above. In many cases, the computations turn out to be exactly the same; however, in some instances, we encounter certain refinements due to the fewer parameters involved. The proofs are substantially identical to those presented in Section \ref{ssc:GL}. Therefore, in the following discussion, we simply state the results for the case where $G=\SL(2,q)$.
\end{remark}

\begin{proposition}
    Let $u\in G$ be a matrix of type 1. Then $\C_G(u)=\C_G(u^p)=G$.
\end{proposition}

\begin{proposition} \label{prop:sl_type2}
    Let $u\in G$ be a matrix of type 2. Then  $\C_G(u^p)\setminus\C_G(u)\neq\emptyset$ if and only if $p\mid q-1$.
\end{proposition}

\begin{proposition} \label{prop:sl_type3}
    Let $u\in G$ be a matrix of type 3. Then  $\C_G(u^p)\setminus\C_G(u)\neq\emptyset$ if and only if $p=q$.
\end{proposition}

\begin{proposition} \label{prop:sl_type4}
    Let $u\in G$ be a matrix of type 4. Then  $\C_G(u^p)\setminus\C_G(u)\neq\emptyset$ if and only if $p\mid q+1$.
\end{proposition}

Combining the previous results, we obtain exactly the same numeric condition as in Theorem \ref{thm:GL_characterization_coloring}. We may therefore strengthen the statement, adjoining the claim related to the special linear group.

\begin{theorem}
    Let $m,p\in\N$ be such that $m\ge 2$ and $p$ is prime with $p\nmid m$. The following are equivalent.
    \begin{enumerate}
            \item The torus knot $\K(m,p)$ is $\conj(\GL(2,q))$-colorable.
            \item The torus knot $\K(m,p)$ is $\conj(\SL(2,q))$-colorable.
            \item $p\mid q(q+1)(q-1)$.
        \end{enumerate}
\end{theorem}

\subsection[Dihedral group]{Dihedral group}

We now proceed to discuss the dihedral groups. We denote the dihedral group of the $n$-gon as ${\D_n}$, and employ the following presentation:
\[
    \D_n=\left\langle r,s\colon r^n=s^2=1,\ srs=r\mo\right\rangle.
\]
Our objective is to derive a numerical characterization for the colorability of $\K(m,p)$ solely in terms of $p$ and $n$.

\begin{theorem}[$\conj(\D_n)$-coloring of Torus Knots] \label{thm:Dn_characterization_coloring}
    Let $m,p\in\N$ be such that $m\ge 2$ and $p$ is prime with $p\nmid m$. The torus knot $\K(m,p)$ is $\conj(\D_n)$-colorable if and only if $p\mid n$.
\end{theorem}

\begin{proof}
    Assume $p\nmid n$. Then there's an element $u=r^{\frac{n}{p}}\in\D_n$ of order $p$. It is well known that 
    \[  
        \zen(\D_n)=\left\langle
        \begin{array}{lr}
            \{1\}  &n\mbox{ odd}\\
            \{1,r^{\frac{n}{2}}\}  &n\mbox{ even}
        \end{array}
        \right..
    \] 
    If $n$ is odd, then trivially $u\not\in\zen(\D_n)$. If $n$ is even, we also have $u\not\in\zen(\D_n)$, because if that was the case, we would have $\frac{n}{2}=\frac{n}{p}$, implying $p=2$, which is excluded. By Remark \ref{rmk:fast_trick}, we conclude that $\K(m,p)$ is $\conj(\D_n)$-colorable.

    Conversely, we assume that $p\nmid n$ and prove that $\K(m,p)$ is not $\conj(\D_n)$-colorable. By Theorem \ref{thm:char_kmp}, it is enough to check that the condition $\C_{\D_n}(u^p)\setminus\C_{\D_n}(u)\neq\emptyset$ never holds. Recall that every element of $u\in\D_n$ may be uniquely expressed as $u=s^tr^k$ for some $t\in\{0,1\}$ and $k\in\{0,\dots,n-1\}$. For $t=1$, then $u$ is an involution, hence $\C_{\D_n}(u^p)=\C_{\D_n}(u)$. Assume $t=0$, that is $u=r^k$. A direct computation shows that
    \[  
        \C_{\D_n}(r^k)=\C_{\D_n}(r^{pk})=\left\langle
        \begin{array}{lr}
            \langle r\rangle  &k\neq\frac{n}{2}\\
            \D_n  &k=\frac{n}{2}
        \end{array}
        \right.
    \] 
    therefore, also in this case, $\C_{\D_n}(u^p)\setminus\C_{\D_n}(u)\neq\emptyset$ does not hold.
\end{proof}

\subsection[Symmetric group]{Symmetric group}

We conclude this section and the paper with a discussion on the symmetric groups. We denote the symmetric group over $n$ letters as ${\S_n}$. Our goal is, again, to derive a numerical characterization for the colorability of $\K(m,p)$ solely in terms of $p$ and $n$.

\begin{theorem}[$\conj(\S_n)$-coloring of Torus Knots] \label{thm:Sn_characterization_coloring}
    Let $m,p\in\N$ be such that $m\ge 2$ and $p$ is prime with $p\nmid m$. The torus knot $\K(m,p)$ is $\conj(\S_n)$-colorable if and only if $p\le n$.
\end{theorem}

\begin{proof}
    If $p\le n$, then $p\mid n!=|\S_n|$, hence $\S_n$ has an element of order $p$, which is not in the centre because $\zen(\S_n)=\{1\}$. Remark \ref{rmk:fast_trick} allows us to conclude that $\K(m,p)$ is $\conj(\S_n)$-colorable.

    Conversely, we assume that $p<n$ and prove that $\K(m,p)$ is not $\conj(\D_n)$-colorable. By Theorem \ref{thm:char_kmp}, it is enough to check that the condition $\C_{\S_n}(u^p)\setminus\C_{\S_n}(u)\neq\emptyset$ never holds. Let $u\in\S_n$ and consider its complete factorization in disjoint cycles $u=\sigma_1\dots\sigma_t$, for some $t\ge 1$. Since disjoint cycles commute, we have $u^p=\sigma_1^p\dots\sigma_t^p$. In particular, since $p<n$ prime, if $\sigma$ is an $r$-cycle, then also $\sigma^p$ is an $r$-cycle. This implies that $u$ and $u^p$ have the same cycle structure, hence they are conjugate in $\S_n$. Consider $\tau\in\s_n$ such that $u^p=\tau u\tau\mo$. Then $\C_{\S_n}(u^p)=\C_{\S_n}(\tau u\tau\mo)=\tau\C_{\S_n}(u)\tau\mo$, hence $|\C_{\S_n}(u^p)|=|\C_{\S_n}(u)|$. Since it always hold that $\C_{\S_n}(u)\le\C_{\S_n}(u^p)$, the two centralisers must be equal, hence $\K(m,p)$ is not $\conj(\S_n)$-colorable if $p<n$.
\end{proof}

\section{Conclusions} \label{sec:conclusions}
We end this paper by posing a few questions for potential future research. Through applying the broad description of torus knots, we have derived characterization theorems for coloring torus knots by employing conjugation quandles with specific groups. Is it possible to extend this approach to additional small groups?

\begin{problem}
    Characterize the conjugation quandle coloring of $\K(m,n)$ using other small groups.
\end{problem}

Moreover, there exists a knot-theoretical tool that enables the association of a polynomial with any given knot, encoding certain properties. This technique is known as the Alexander polynomial (see \cite{murasugi1996knot}). The Alexander polynomials of torus knots are well-understood and easy to manipulate, involving specific divisibility conditions on their parameters. It's only natural to inquire whether there is a correlation between Alexander polynomials and colorings.

\begin{problem}
    What are the relations (if any) between the conjugation quandle coloring of $\K(m,n)$ and its Alexander polynomial?
\end{problem}

Furthermore, a well-known family of satellite knots is the one consisting of Whitehead doubles. For a given knot $K$, its Whitehead double $\W(K)$ is constructed by duplicating its arcs and introducing two additional crossings (see \cite{Jang}). Is it feasible to formulate a characterization theorem for the quandle colorability of $\W(\K(m,n))$ using a strategy akin to what we achieved in Theorem \ref{thm:char_kmp}? Ideally, this approach would begin by simplifying the task, initially omitting divisors and following a similar pattern as seen in Theorem \ref{thm:kmn_iff_kpq}.

\begin{problem}
    Characterize the conjugation quandle colorability of the Whitehead double of $\K(m,n)$.
\end{problem}

\begin{acknowledgements}
    This paper is built upon research carried out during the Ph.D. studies of the author, who received partial support from both the GAUK grant (301-10/252012) and Z. Patáková's Primus grant (301-45/247107). Furthermore, the author wishes to express gratitude to their Ph.D. advisor, David Stanovský, for his patient guidance and inquisitive approach. Additionally, appreciation is extended to Petr Vojtěchovský for offering valuable insights that have enriched specific prior results.
\end{acknowledgements}

\nocite{*}
\printbibliography

@manual{GAP4,
    key          = "GAP",
    author = "GAP",
    organization = "The {GAP~Group}",
    title        = "{GAP -- Groups, Algorithms, and Programming, Version 4.12.2}",
    year         = 2022,
    url = {https://www.gap-system.org} 
}

@article {Zhou,
    AUTHOR = {Zhou, Boxin and Liu, Ximin},
     TITLE = {Quandle coloring quivers of {$(p,3)$}-torus links},
   JOURNAL = {J. Knot Theory Ramifications},
  FJOURNAL = {Journal of Knot Theory and its Ramifications},
    VOLUME = {32},
      YEAR = {2023},
    NUMBER = {3},
     PAGES = {Paper No. 2350016, 23},
      ISSN = {0218-2165,1793-6527},
   MRCLASS = {57K12 (57K10)},
  MRNUMBER = {4581230},
       DOI = {10.1142/S0218216523500165},
       URL = {https://doi.org/10.1142/S0218216523500165},
}

@article {Basi,
    AUTHOR = {Basi, Jagdeep and Caprau, Carmen},
     TITLE = {Quandle coloring quivers of {$(p,2)$}-torus links},
   JOURNAL = {J. Knot Theory Ramifications},
  FJOURNAL = {Journal of Knot Theory and its Ramifications},
    VOLUME = {31},
      YEAR = {2022},
    NUMBER = {9},
     PAGES = {Paper No. 2250057, 14},
      ISSN = {0218-2165,1793-6527},
   MRCLASS = {57K10 (57K12)},
  MRNUMBER = {4475496},
MRREVIEWER = {Leandro\ Vendramin},
       DOI = {10.1142/S0218216522500572},
       URL = {https://doi.org/10.1142/S0218216522500572},
}

@article {Asami,
    AUTHOR = {Asami, Soichiro and Kuga, Ken'ichi},
     TITLE = {Colorings of torus knots and their twist-spuns by {A}lexander
              quandles over finite fields},
   JOURNAL = {J. Knot Theory Ramifications},
  FJOURNAL = {Journal of Knot Theory and its Ramifications},
    VOLUME = {18},
      YEAR = {2009},
    NUMBER = {9},
     PAGES = {1259--1270},
      ISSN = {0218-2165,1793-6527},
   MRCLASS = {57M25 (57Q45)},
  MRNUMBER = {2569560},
MRREVIEWER = {Gabriela\ Hinojosa},
       DOI = {10.1142/S0218216509007452},
       URL = {https://doi.org/10.1142/S0218216509007452},
}

@book{elhamdadi2015quandles,
  title={Quandles},
  author={Elhamdadi, Mohamed and Nelson, Sam},
  volume={74},
  year={2015},
  publisher={American Mathematical Soc.}
}

@book{murasugi1996knot,
  title={Knot Theory and Its Applications},
  author={Murasugi, K.},
  isbn={9783764338176},
  lccn={96016329},
  url={https://books.google.cz/books?id=n1DvAAAAMAAJ},
  year={1996},
  publisher={Birkh{\"a}user Boston}
}

@article {Kuperberg,
    AUTHOR = {Kuperberg, Greg},
     TITLE = {Knottedness is in NP, modulo {GRH}},
   JOURNAL = {Adv. Math.},
  FJOURNAL = {Advances in Mathematics},
    VOLUME = {256},
      YEAR = {2014},
     PAGES = {493--506},
      ISSN = {0001-8708,1090-2082},
   MRCLASS = {57M25 (11M26 68Q25)},
  MRNUMBER = {3177300},
MRREVIEWER = {Jonathan\ Spreer},
       DOI = {10.1016/j.aim.2014.01.007},
       URL = {https://doi.org/10.1016/j.aim.2014.01.007},
}

@article{Jang,
title = {The canonical genus for Whitehead doubles of a family of alternating knots},
journal = {Topology and its Applications},
volume = {159},
number = {17},
pages = {3563-3582},
year = {2012},
issn = {0166-8641},
doi = {https://doi.org/10.1016/j.topol.2012.08.017},
url = {https://www.sciencedirect.com/science/article/pii/S0166864112003409},
author = {Hee Jeong Jang and Sang Youl Lee}
}

@article {Clark,
    AUTHOR = {Clark, W. Edwin and Elhamdadi, Mohamed and Saito, Masahico and
              Yeatman, Timothy},
     TITLE = {Quandle colorings of knots and applications},
   JOURNAL = {J. Knot Theory Ramifications},
  FJOURNAL = {Journal of Knot Theory and its Ramifications},
    VOLUME = {23},
      YEAR = {2014},
    NUMBER = {6},
     PAGES = {1450035, 29},
      ISSN = {0218-2165,1793-6527},
   MRCLASS = {57M25 (57M27)},
  MRNUMBER = {3253967},
MRREVIEWER = {Blake\ Mellor},
       DOI = {10.1142/S0218216514500357},
       URL = {https://doi.org/10.1142/S0218216514500357},
}

@article {Iwakiri,
    AUTHOR = {Iwakiri, Masahide},
     TITLE = {Quandle cocycle invariants of torus links},
 BOOKTITLE = {Intelligence of low dimensional topology 2006},
    SERIES = {Ser. Knots Everything},
    VOLUME = {40},
     PAGES = {57--64},
 PUBLISHER = {World Sci. Publ., Hackensack, NJ},
      YEAR = {2007},
      ISBN = {978-981-270-593-8},
   MRCLASS = {57M25 (57M27)},
  MRNUMBER = {2371709},
       DOI = {10.1142/9789812770967\_0008},
       URL = {https://doi.org/10.1142/9789812770967_0008},
}

@article {Fish,
    AUTHOR = {Fish, Andrew and Lisitsa, Alexei and Stanovsk\'{y}, David and
              Swartwood, Sarah},
     TITLE = {Efficient knot discrimination via quandle coloring with {SAT}
              and \#{SAT}},
 BOOKTITLE = {Mathematical software---{ICMS} 2016},
    SERIES = {Lecture Notes in Comput. Sci.},
    VOLUME = {9725},
     PAGES = {51--58},
 PUBLISHER = {Springer, [Cham]},
      YEAR = {2016},
      ISBN = {978-3-319-42432-3},
   MRCLASS = {68T20 (57M27)},
  MRNUMBER = {3662298},
       DOI = {10.1007/978-3-319-42432-3},
       URL = {https://doi.org/10.1007/978-3-319-42432-3},
}

\end{document}